\documentclass[a4paper, 12pt, onecolumn]{article}
\usepackage{geometry}
\usepackage{amsmath}
\usepackage{amsthm}
\usepackage{amsfonts}
\usepackage{bbm}
\usepackage{CJK}
\usepackage{fancyhdr}
\usepackage{graphicx}
\usepackage{psfrag}
\usepackage{amsfonts,amsmath,amsthm, amssymb}
\usepackage{latexsym, euscript, epic, eepic}
\usepackage{time}
\usepackage{txfonts}
\usepackage{colortbl}
\usepackage{stmaryrd}
\usepackage{mathrsfs}
\usepackage{txfonts}
\usepackage{amsfonts}
\usepackage{color}
\usepackage{lineno}
\usepackage[square, comma, sort&compress, numbers]{natbib}

\usepackage{indentfirst,latexsym,bm}

 \numberwithin{equation}{section}

\begin{document}
\newtheorem{theorem}{Theorem}[section]
\newtheorem{proposition}[theorem]{Proposition}
\newtheorem{remark}[theorem]{Remark}
\newtheorem{corollary}[theorem]{Corollary}
\newtheorem{definition}{Definition}[section]
\newtheorem{lemma}[theorem]{Lemma}
\newcommand{\wuhao}{\fontsize{10pt}{10pt}\selectfont}

\title{Inequalities for Zero-Balanced Gaussian hypergeometric function$^\star$  \footnotetext{$\star$ This research is
supported by the Natural Science Foundation  of P. R. China (No.11401531, No.11601485.) and the Natural Science Foundation of Zhejiang Province (No.Q17A010038).}
\footnotetext{*Corresponding author.} \footnotetext{E-mail address:
htiren@zstu.edu.cn, mxy@zstu.edu.cn,xiaohui.zhang@zstu.edu.cn.}}
\author{\small
Ti-Ren Huang*, Xiao-Yan Ma and Xiao-Hui Zhang\\
\small{(School of Science, Zhejiang Sci-Tech University, Hangzhou
310018, China)}}
\date{}
\maketitle
 \fontsize{12}{22}\selectfont\small
 \paragraph{Abstract:}In this paper, we consider the monotonicity of certain combinations of the Gaussian hypergeometric functions $F(a-1,b;a+b;1-x^c)$ and $F(a-1-\delta,b+\delta;a+b;1-x^d)$ on $(0,1)$ for $\delta\in(a-1,0)$, and study the problem of comparing these two functions, thus  get the largest value $\delta_1=\delta_1(a,c,d)$ such that  the inequality  $F(a-1,b;a+b;1-x^c)<F(a-1-\delta,b+\delta;a+b;1-x^d)$ holds for all $x\in (0,1)$.
\\[10pt]
\\[10pt]

\section{\normalsize Introduction}\label{sec:bd}
In this paper we consider the Gaussian hypergeometric function
\begin{align}\label{hypergeometric function}
F(a,b;c;x)= _2F_1(a,b;c;x)=\sum_{n=0}^\infty\frac{(a,n)(b,n)}{(c,n)}\frac{x^n}{n!},
\end{align}
for  $x\in (-1,1)$, where $(a,n)$ denotes the shifted factorial function  $(a,n)\equiv a(a+1)\cdots (a+n-1)$, $n=1,2,\cdots$, and $(a,0)=1$ for $a\neq 0$.
It is well known that the function $F(a, b; c; x)$ has many important applications in geometric function theory, theory of mean values, and in several other contexts, and many classes of elementary functions and special functions in mathematical physics are particular or limiting cases of this function \cite{Anderson G D1, Anderson G D2,Anderson G D3,Baricz' A,Barnard R W,Borwein J M2,Neuman E,Ponnusamy S,Ponnusamy S2,Qiu S L}.

For $r\in(0,1)$ and $a\in (0,1)$, the generalized elliptic integrals of the first and second kinds are defined as
\begin{align*}
\mathscr{K}_a(r)=\frac{\pi}{2}F(a,1-a;1;r^2),\quad \mathscr{E}_a(r)=\frac{\pi}{2}F(a-1,1-a;1;r^2).
\end{align*}
In the particular case $a=1/2$, the generalized elliptic integrals reduce to the complete elliptic integrals
\begin{align*}
\mathscr{K}(r)=\frac{\pi}{2}F(\frac12,\frac12;1;r^2),\quad \mathscr{E}(r)=\frac{\pi}{2}F(-\frac12,\frac12;1;r^2).
\end{align*}

J. M. Borwein and P. B. Borwein, in order to find out the connections between the arithmetic-geometric mean value and other mean values, showed in their paper  \cite{Borwein J M2} that
\begin{align*}
F(\frac12,\frac12;1;1-x^2)<F(\frac12-\delta,\frac12+\delta;1;1-x^3),
\end{align*}
for $\delta=1/6$ and $x\in (0,1)$.

Subsequently, it was proved by Anderson et al. in \cite{Anderson G D1}  that
\begin{align}\label{inequalitis1}
F(\frac12,\frac12;1;1-x^c)<F(\frac12-\delta,\frac12+\delta;1;1-x^d)<F(\frac12,\frac12;1;1-x^d),
\end{align}
for all $x\in(0,1)$, $c,d\in(0,\infty)$ with $0<4c<\pi d<\infty$ and $\delta\in (0,\delta_0)$ where $\delta_0=[(d\pi-4c)/(4\pi d)]^\frac12$.
It was conjectured  for $c=2,d=3$  that the best value of $\delta_0$ for which (\ref{inequalitis1}) is valid is
\begin{align*}
\delta_0=\frac{1}{\pi}\arccos\frac23\approx0.268.
\end{align*}

In \cite{Anderson G D3}, Anderson et al. considered the more general case of (\ref{inequalitis1}). They showed several monotonicity theorems of certain combinations of $F(a,b;a+b;1-x^c)$ and $F(a-\delta,b+\delta;a+b;1-x^d)$ on $(0,1)$ for given $a,b,c,d \in(0,\infty),a\leq b$ and $c\leq d$, and found $\sup\{\delta\in(0,a)|F(a,b;c;1-x^c)<F(a-\delta,b+\delta;a+b;1-x^d)\ \hbox{for}\  x\in(0,1)\}$. Thus the above conjecture and the following open problem raised in \cite{Anderson G D3} were answered.

\textbf{Open problem}. Is it true, for small values of $\delta$, say $0<\delta<\min\{a,b\}$, that
\begin{align*}
F(a, b; a+b; 1-x^c)<F(a-\delta, b+\delta; a+b; 1-x^d),
\end{align*}
for $x\in (0, 1), a,b,c,d\in (0, \infty)$ with $0 < c < d < \infty$?

Motivated by the results mentioned above,  the following question was naturally raised.

\textbf{Question}. What is the best value of $\delta_1=\delta(a,c,d)\in(a-1,0)$ such that
\begin{align*}
F(a-1,b;a+b;1-x^c)<F(a-1-\delta,b+\delta;a+b;1-x^d),
\end{align*}
for $x\in(0,1), a\in(0,1), b\geq1-a$ and $0<c<d<\infty$.

In \cite{Song}, Song et al. established a monotonicity theorem of certain combinations of $F(-1/2, 1/2; $ $1; 1-x^c)$ and $F(-1/2-\delta,1/2+\delta;1;1-x^d)$ on $(0,1)$ for given  $0<c\leq 5d/6$, and got the following inequality:  For $\delta_1=(\sqrt{c/d}-1)/2$ and $\delta\in (-1/2,\delta_1)$,
\begin{align}\label{inequ}
F(-\frac12,\frac12;1;1-x^c)< F(-\frac12-\delta,\frac12+\delta;1;1-x^d).
\end{align}
and $\delta_1=(\sqrt{c/d}-1)/2$ is the largest value for the inequality (\ref{inequ}) holds for all $x\in(0,1)$,

 Besides, they also considered monotonicity property of certain combinations of $F(a-1-\delta, 1-a+\delta; 1; 1-x^3)$ and $F(a-1, 1-a; 1; 1-x^2)$ for given $a\in[1/29, 1)$ and $\delta\in (a-1, 0)$, and found the largest value $\delta_1$ such that inequality $F(a-1,1-a; 1; 1-x^2)<F(a-1-\delta,1-a+\delta;1;1-x^3)$ holds for all $x\in(0,1)\}$.

In this paper, we will show a monotonicity theorem of certain combinations of $F(a-1, b; a+b; 1-x^c)$ and $F(a-1-\delta,b+\delta; a+b; 1-x^d)$ on $(0,1)$,  and find the largest value $\delta_1=\delta_1(a,c,d)$ such that the inequality  $F(a-1,b;a+b;1-x^c)<F(a-1-\delta,b+\delta;a+b;1-x^d)$ holds for all $x\in (0,1)$.
  Throughout this paper, we shall always let $a\in (0,1), b\geq 1-a$, and
\begin{align*}
&\alpha=a(b+1),\quad \beta=b(1-a),\quad p=\alpha+\beta =a+b,\\
&h=\alpha\beta (p+\beta )=a(1-a)b(b+1)(a+2b-ab),\\
&k=\beta (p+1)+p=b(1-a)(a+b+1)+a+b.
\end{align*}
The main results are stated as follows.
\begin{theorem}\label{main theorem1}
Assume that $a\in (0,1), b\geq 1-a$,  and $\alpha ,\beta ,p,h$ satisfy either $\alpha \geq \sqrt3\beta $ or $\alpha <\sqrt3\beta $, and  $4h(\beta +p)\geq p^4$. Let $0< c/d \leq (\beta +p)/k$, and $\delta_1 $ be the large root of $(c/d-1)\beta+(a-b-1)\delta -\delta^2=0$, namely $\delta_1=[(a-b-1)+\sqrt{(p-1)^2+4\beta c/d}]/2<0$. We have that

(1) If $\delta\in(a-1,\delta_1]$, the function
\begin{align*}
G(x)=\frac{F(a-1-\delta,b+\delta;p;1-x^d)-F(a-1,b;p;1-x^c)}{1-x^c}
\end{align*}
is strictly decreasing from $(0,1)$ onto $(C_1(\delta),C_2(\delta))$, where
\begin{align*}
&C_1(\delta)=\frac d{pc}\left((\frac cd-1)\beta +(a-b-1)\delta-\delta^2\right)\geq 0 \\
&C_2(\delta)=\frac{1}{pB(a-\delta,b+1+\delta)}-\frac{1}{p B(a,b+1)}.
\end{align*}
In particular, for all $x\in(0,1)$ and $\delta\in (a-1,\delta_1]$,
\begin{align*}
F(a-1,b;p;1-x^c)+C_1(\delta)(1-x^c)&<F(a-1-\delta,b+\delta;p;1-x^d)\nonumber\\
&<F(a-1,b;p;1-x^c)+C_2(\delta)(1-x^c).
\end{align*}

(2) If $\delta_1<\delta<0$, as the functions of $x$, $F(a-1-\delta,b+\delta;p;1-x^d)$ and $F(a-1,b;p;1-x^c)$ are not directly comparable on $(0,1)$, that is, neither
\begin{align*}
F(a-1,b;p;1-x^c)<F(a-1-\delta,b+\delta;p;1-x^d),
\end{align*}
nor its reversed inequality holds for all $x\in(0,1)$.
\end{theorem}


The following Theorem can be directly obtained by Theorem \ref{main theorem1}.
\begin{theorem}\label{main theorem3}
Assume that $a\in (0,1), b\geq 1-a$. Let $\alpha, \beta, p, h $ be as in Theorem \ref{main theorem1},  if $0< c/d \leq (\beta +p)/k$, then
\begin{align*}
&\sup\{\delta\in(a-1,0)|F(a-1,b;a+b;1-x^c)<F(a-1-\delta,b+\delta;a+b;1-x^d),\nonumber \\
&\hbox {for all}\  x\in (0,1)\}=\frac{a-b-1+\sqrt{(p-1)^2+4\alpha c/d}}{2}.
\end{align*}
\end{theorem}

\section{\normalsize Preliminaries}\label{sec:bd}
Before we prove our main results stated in Section 1, we need to
establish several technical lemmas. Firstly, let us recall some known
results for $F(a,b;c;x)$ and for the gamma function.

For $x>0,y>0$,  the Euler gamma function $\Gamma(x)$, its logarithmic derivative $\Psi(x)$
and the beta function $B(x,y)$ are defined as
\begin{align*}
\Gamma(x)=\int_0^\infty t^{x-1}e^{-t}dt, \quad \Psi(x)=\frac{\Gamma'(x)}{\Gamma(x)}, \quad B(x,y)=\frac{\Gamma(x)\Gamma(y)}{\Gamma(x+y)}
\end{align*}
respectively (c.f.\cite{WW}).
The gamma function satisfies the difference equation (\cite{WW}, p. 237)
\begin{align*}
\Gamma(x+1)=x\Gamma(x),
\end{align*}
if $x$ is not a nonpositive integer and has the so-called  reflection property (\cite{WW}, p. 239)
\begin{align}
\Gamma(x)\Gamma(1-x)=\frac{\pi}{\sin(\pi x)}.
\end{align}
We shall also need an asymptotic formula of gamma function (\cite{Qiu S L2}, p. 628)
\begin{align}\label{asymptotic formula}
\frac{\Gamma(n+a)}{\Gamma(n+b)}\sim n^{a-b}, \quad n\rightarrow+\infty, n\in \mathbb{N}.
\end{align}
The hypergeometric function (\ref{hypergeometric function}) has the following difference formula (\cite{Qiu S L2}),
\begin{align*}
\frac{dF(a,b;c;x)}{dx}=\frac{ab}{c}F(a+1,b+1;c+1;x)
\end{align*}
and the asymptotic limit (\cite{Qiu S L2}, p. 630),
\begin{align*}
lim_{x\rightarrow 1^{-}}F(a,b;c;x)=\frac{\Gamma(c)\Gamma(c-a-b)}{\Gamma(c-a)\Gamma(c-b)},\quad c>a+b.
\end{align*}

 The following Lemma can be find  Lemma 2.1.5 in \cite{Qiu S L}, and Lemma 2.11 in \cite{Anderson G D3}, respectively.

\begin{lemma}\label{lemma0}
(1) For $a,b,c,d\in (0,\infty)$, the function $x\rightarrow(1-x)^dF(a,b;c;x)$ is (strictly) decreasing on $(0,1)$ if and only if $d\geq \max\{a+b-c,ab/c\}$ ($d > \max\{a+b-c,ab/c\}$).



(2) For $a, b\in (0, \infty)$  with $a\leq b$, the function $x\rightarrow B(a-x,b+x)$ is strictly increasing and convex on
$(0,a)$.


\end{lemma}

\begin{lemma}\label{a_0}
Let $a\in(0,1), b\geq 1-a$, and $\alpha ,\beta , p, h, k$ be as in Section 1, a function $g_1(y)$ is defined as
\begin{align*}
g_1(y)=y^2+\frac{p^2}{k}y+\frac{h(p+\beta )}{k^2}.
\end{align*}
Then,
(1) if $\alpha\  \hbox{and}\ \beta $ satisfy $\alpha \geq \sqrt3\beta $, $g_1(y)$
is an increasing function from $(-h/(\alpha k),0)$ onto $((1-a)^2(p^2+p\beta )/k^2, h/k^2)$.

(2) If $\alpha , \beta , p\ \hbox{and} \ h$ satisfy $\alpha <\sqrt3\beta $, and $4h(\beta +p)\geq p^4$, we have $g_1(y)\geq 0$ for all $(-h/(\alpha k),0)$.
\end{lemma}
\begin{proof}
(1) Clearly,
\begin{align*}
g_1(-{h}/{(\alpha k)})={(1-a)^2(p^2+p\beta )}/{k^2}>0, \quad g_1(0)=h/k^2.
\end{align*}
Since  $\alpha \geq \sqrt3\beta $,
\begin{align*}
\frac{p^2}{2k}-\frac{h}{\alpha k}=\frac{\alpha ^2 -3\beta ^2}{2k}\geq 0,
\end{align*}
hence, for $y\in (-h/(\alpha k),0)$,
\begin{align*}
g_1'(y)=2y+{p^2}/{k}=\frac{p^2}{k}-\frac{2h}{\alpha k}\geq 0,
\end{align*}
and  $g_1(y)$ is an  increasing function.

(2) For $\alpha <\sqrt3\beta $, and $4h(\beta +p)\geq p^4$,  we have
\begin{align*}
g_1(y)=\left(y+\frac{p^2}{2k}\right)^2+\frac{h(\beta +p)}{k^2}-\frac{p^4}{4k^2}\geq 0.
\end{align*}
hence,  $g_1(y)\geq 0$ for all $(-h/(\alpha k),0)$.
\end{proof}

\begin{remark} \label{remark 1}
Let  $a\in (0,1), b\geq 1-a$, and $\alpha ,\beta , p\ \hbox{and}\  h$ be as in Section 1, we have that
\begin{align*}
(1)\ \  &\alpha\leq \sqrt3\beta \Leftrightarrow \sqrt 3/a-1/b\geq 1+\sqrt 3.\\
(2)\ \  &\alpha< \sqrt3\beta,\ 4h(\beta +p)\geq p^4 \Leftrightarrow\sqrt 3/a-1/b> 1+\sqrt 3,\\
& \quad \qquad\qquad\qquad\qquad4a(1-a)b(b+1)(a+2b-ab)\geq (a+b)^4.
\end{align*}
\end{remark}

\begin{lemma}\label{g(x,y)}
Let $a\in(0,1), b\geq 1-a$, $\alpha ,\beta , p\ \hbox{and}\  h$ be as in Section 1, and $D=\{(x,y)|0<x<(\beta +p)/k,-\beta x<y<0\}$. Define  the function $g(x,y)$ on the domain $D$ as
\begin{align*}
g(x,y)=y^2+((p+1)x-1)y+\alpha \beta  x^2.
\end{align*}
If $\alpha ,\beta ,p\ \hbox{and}\ h$ satisfy either $\alpha \geq \sqrt3\beta $, or $\alpha <\sqrt3\beta $ and  $4h(\beta +p)\geq p^4$, then $\inf_{(x,y)\in h}g(x,y)=0$.
\end{lemma}
\begin{proof}
By differentiation,
\begin{align*}
\frac{\partial g(x,y)}{\partial x}=(p+1)y+2\alpha \beta  x,\quad \frac{\partial g(x,y)}{\partial y}=2y+(p+1)x+1.
\end{align*}
Let ${\partial g(x,y)}/{\partial x}={\partial g(x,y)}/{\partial y}=0$, we have
\begin{align*}
x_0=\frac{p+1}{(p+1)^2-4\alpha \beta },\quad y_0=-\frac{2\alpha \beta }{(p+1)^2-4\alpha \beta },\quad g(x_0,y_0)=\frac{\alpha \beta }{(p+1)^2-4\alpha \beta }>0.
\end{align*}

On the other hand,
\begin{align*}
&g(x,0)=\alpha \beta x^2\geq 0\quad \qquad\quad\hbox{for}\quad 0<x<(\beta +p)/k,\\
&g(x,-\beta x)=\beta x(1-x)>0 \quad \hbox{for}\quad 0<x<(\beta +p)/k<1.
\end{align*}
Since
\begin{align*}
g((\beta +p)/k,y)=y^2+\frac{p^2}{k}y+\frac{h(p+\beta )}{k^2}=g_1(y), \quad y\in (-h/(\alpha k),0),
\end{align*}
we get $g((\beta +p)/k,y)\geq 0$ for all  $y\in (-h/(\alpha k),0)$ by Lemma \ref{a_0}, hence  $\inf_{(x,y)\in h}g(x,y)=0$.
\end{proof}

Since $b\geq 1-a$, we have the following Lemma.
\begin{lemma}\label{lemmaf}
Let $a\in (0,1), b\geq 1-a, u=a-\delta, v=a+\delta\ \hbox{and}\  \delta\in (a-1,0)$, we have

(1) the function $f_1(\delta)=uv+u-1+\beta =p-1+(a-b-1)\delta-\delta^2$ is strictly decreasing from $(a-1,0)$ onto $(p-1,p-1+\alpha)$.

(2) the function $f_2(\delta)=u(v+1)=\alpha +(a-b-1)\delta-\delta^2$ is strictly decreasing from $(a-1,0)$ onto $(\alpha , p)$.

(3) the function $f_3(\delta)=v(u-1)=-\alpha+(a-b-1)\delta-\delta^2$ is strictly decreasing from $(a-1,\delta_1)$ onto $(- c\beta/d  , 0)$, where $\delta_1$ is as in Theorem \ref{main theorem1}.
\end{lemma}

\begin{lemma}\label{a_0,a_1}
The function $f_4(a)=4a(2-a)(1-a)^2(a^2-2a+2)^2-1$  has  only two null points $a_0\in(1/32,1/31), a_1\in(41/50,42/50)$ in $(0,1)$.
\end{lemma}
\begin{proof}
Since $f_4(0)=-1, f_4(1/2)=299/64, f_4(1)=-1$ and $f_4(x)$ has at least two  null points in $(0,1)$. Assume that $f_4(x)$ has more than two null points in (0,1), then $f_4'(x)$ has more than two  null points in (0,1) by Rolle mean value theorem. But,
\begin{align*}
f_4'(a)=-8(a-1)(a^2-2a-2)(4a^4-16a^3+15a^2+2a-2),
\end{align*}
$a-1<0, a^2-2a-2<0$,  it is easy to know that $f_5(a)=4a^4-16a^3+15a^2+2a-2$ is an increasing function in $(0,1)$, $f_5(0)=-2\ \hbox{and}\
  f_5(1)=3$, hence $f_5(a)$  has only one root in $(0,1)$, Contradiction. By elementary computations, $f_4(1/32)<0, f_4(1/31)>0, f_4(41/50)>0\ \hbox{and}\ f_4(42/50)<0,$ so there exist two null points  $a_0\in(1/32,1/31)\ \hbox{and}\  a_1\in(41/50,42/50)$ in $(0,1)$.
\end{proof}

\begin{lemma}\label{Q(n)}
If $a\in (0,1),\ b\geq 1-a, \  0< c/d\leq (\beta +p)/k, \ \delta\in(a-1,0)$ and $n\in\mathbb{N}$, let $u=a-\delta, v=b+\delta$, then
\begin{align*}
Q(n)=\frac{\Gamma(u+n-1)\Gamma(v+n)}{\Gamma(a+n-1)\Gamma(b+n)}\left\{(\frac cd-1)(u+v+n)+u(v+1)\right\}
\end{align*}
is strictly decreasing and $\lim_{n\rightarrow \infty}Q(n)=-\infty$.
\end{lemma}
\begin{proof}
 By computation, we have
 \begin{align*}
 Q(n+1)-Q(n)=\frac{\Gamma(n+u-1)\Gamma(n+v)}{\Gamma(n+a)\Gamma(b+n+1)}Q_1(n),
 \end{align*}
 where
 \begin{align*}
 Q_1(n)&=\left( c/d-1\right)n^2-\left(c/d-1\right)(uv+u-1+\beta )n+A=\left( c/d-1\right)n^2-\left(c/d-1\right)f_1(\delta)n+A,\\
  A &=\left( c/d-1\right)(u+v+1)+u(v+1)v(u-1)-\beta (\left( c/d-1\right)(u+v)+u(v+1)).
 \end{align*}
 Since $\delta\in(a-1,0)$ and $f_1(\delta)\geq f_1(0)=p-1\geq 0$.
 Hence, $Q_1(n)$ is strictly decreasing and
 \begin{align*}
 Q_1(n)\leq Q_1(1)&=(u(v+1))^2+[(c/d-1)(2+p)-\alpha]u(v+1)
 -(c/d-1)\alpha (p+1)\\
 &=f_2(\delta)^2+[(c/d-1)(2+p)-\alpha]f_2(\delta)
 -(c/d-1)\alpha (p+1)\\
 &=:F(f_2(\delta)).
 \end{align*}
Since $f_2(\delta)$
is strictly decreasing from $(a-1,0)$ onto $(\alpha, p)$ by Lemma \ref{lemmaf} and $0< c/d\leq (\beta+p)/k<1$, we have
\begin{align*}
F(\alpha ))=(c/d-1)\alpha <0,\quad F(p)=c/(dk)-(\beta +p)<0.
\end{align*}
Hence, it is easy to know that $Q_1(n)<0$ for $n\in \mathbb{N}$, and the monotonicity of $Q(n)$ follows.
Moreover, by (\ref{asymptotic formula}), we have
\begin{align*}
\lim_{n\rightarrow \infty}Q(n)=\lim_{n\rightarrow \infty}\left[(\frac dc-1)(n+u+v)+u(v+1)\right]=-\infty.
\end{align*}
\end{proof}

\begin{lemma}\label{monotone}
For $-\infty <a<b<\infty$, let $f,g: [a,b]\rightarrow R$ be continuous on
$[a, b]$, and be differentiable on $(a, b)$. Let $g'(x) \neq 0$ on $(a, b)$. If $f'(x)/g'(x)$ is increasing (decreasing) on
$(a, b)$, then so are
\begin{align*}
\frac{f(x)-f(a)}{g(x)-g(a)} \quad \hbox{and}\quad \frac{f(x)-f(b)}{g(x)-g(b)}.
\end{align*}
If $f'(x)/g'(x)$ is strictly monotone, then the monotonicity in the conclusion is also strict.
\end{lemma}

\section{\normalsize Proof of the main theorem}\label{sec:bd}

\emph{Proof of  Theorem \ref{main theorem1}}. Let $u=a-\delta, v=b+\delta$ and $t=1-(1-x)^{d/c}$, we obtain that
\begin{align}\label{G1}
G_1(x)=G((1-x)^{\frac 1c})&=\frac1x[F(a-1-\delta,b+\delta;p;t)-F(a-1,b;p;x)]\nonumber\\
&=\frac 1x[F(u-1,v;p;t)-F(a-1,b;p;x)],
\end{align}
we let $f(x)=F(u-1,v;p;t)-F(a-1,b;p;x)$ and $g(x)=x$, then $G_1(x)=f(x)/g(x)$ and $f(0)=g(0)=0$.
\begin{align}\label{f'}
\frac{f'(x)}{g'(x)}=f'(x)&=\frac dc (1-x)^{(d/c)-1}\frac{v(u-1)}{p}F(u,v+1;p+1;t)\nonumber\\
&\quad+\frac{\beta }{p}F(a,b+1;p+1;x),
\end{align}
and
\begin{align}\label{f''(x)}
f''(x)&=-\frac{v(u-1)d}{cp} \left(\frac dc-1\right)(1-x)^{ (d/c)-2}F(u,v+1;p+1;t)\nonumber\\
&\quad+\frac{u(u-1) v(v+1)d^2}{p(p+1)c^2}(1-x)^{2[(d/c)-1]}F(u+1,v+2;p+1,t)\nonumber\\
&\quad+\frac{\alpha \beta }{p(p+1)}F(a+1,b+2;p+2;x).
\end{align}

The desired
monotonicity of $G_1(x)$ will follow from Lemma \ref{monotone} if we can prove that$f'(x)$ is increasing on $(0,1)$
or $f''(x)>0$ on $(0,1)$. It is easy to know that  $x\rightarrow (1-x)^{1/c} (c>0)$ is strictly decreasing on $(0,1)$.  Let
\begin{align*}
h(t)&=- \frac{v(u-1)d}{cp}(\frac dc-1)(1-t)F(u,v+1;p+1;t)\nonumber\\
&\quad+\frac{u(u-1)v(v+1)d^2}{p(p+1)c^2}(1-t)^2F(u+1,v+2;p+2,t).
\end{align*}
Since $(1-t)=(1-x)^{(d/c)}$,
then it follows from (\ref{f''(x)}) that
 \begin{align}\label{(1-x)^2f''}
 (1-x)^2 f''(x)=h(t)+\frac{\alpha \beta }{p(p+1)}(1-x)^2F(a+1,b+2;p+2;x).
 \end{align}
Using the series expansion for $F(a,b;c;x)$, we have
 \begin{align}\label{h(t)}
 h(t)=&\frac{d^2}{c^2}(1-t)\Big[(\frac cd-1)\frac{v (u-1)}{p}\sum_{n=0}^\infty\frac{(u,n)(v+1,n)}{(p+1,n)}\frac{t^n}{n!}\nonumber\\
 &\quad+(1-t)\frac{u(u-1)v(v+1)}{p(p+1)}\sum_{n=0}^\infty\frac{(u+1,n)(v+2,n)}{(p+2,n)}\frac{t^n}{n!}\Big],\nonumber\\
 &=\frac{d^2}{c^2}(1-t)\Big[(\frac cd-1)\sum_{n=0}^\infty\frac{(u-1,n+1)(v,n+1)}{(p,n+1)}\frac{t^n}{n!}\nonumber\\
 &\quad+(1-t)\sum_{n=0}^\infty\frac{(u-1,n+2)(v,n+2)}{(p,n+2)}\frac{t^n}{n!}\Big],\nonumber\\
 &=\frac{d^2}{c^2}(1-t)\sum_{n=0}^\infty\frac{(u-1,n+1)(v,n+1)}{(p,n+2)}\Big[(\frac cd-1)(p+n+1)\nonumber\\
 &\quad+(u+n)(v+n+1)-n(p+n+1)\Big]\frac{t^n}{n!}.
 \end{align}
 Since $a\in(0,1)$ and $b\geq 1-a$,  $2>\max\{ a+1+b+2-(a+b+2),[(a+1)(b+2)]/(a+b+2)\}$,
 we have  that $(1-x)^2 F(a+1,b+2;a+b+2;x)$ is strictly decreasing on $(0,1)$ by Lemma \ref{lemma0}(1).  While
$t/x = [1-(1-x)^{d/c}]/x$ is strictly decreasing from $(0, 1)$ onto $(1, d/c)$. Thus, $t>x$  and  the following inequality holds
\begin{align}\label{(1-x)^2F}
(1-x)^2 F(a+1,b+2;a+b+2;x)>(1-t)^2 F(a+1,b+2;a+b+1;t).
\end{align}
By the series expansion of $F(a,b;c;x)$, we obtain that
\begin{align}\label{(1-t)F}
\frac{\alpha \beta }{p(p+1)}(1-t)& F(a+1,b+2;p+2;t)
=(1-t)\sum_{n=0}^\infty\frac{(a-1,n+2)(b,n+2)}{(p,n+2)}\frac{t^n}{n!}\nonumber\\
&=\sum_{n=0}^\infty\frac{(a-1,n+2)(b,n+2)}{(p,n+2)}\frac{t^n}{n!}-\sum_{n=0}^\infty\frac{(a-1,n+2)(b,n+2)}{(p,n+2)}\frac{t^{n+1}}{n!}\nonumber\\
&=\sum_{n=0}^\infty\frac{(a-1,n+2)(b,n+2)}{(p,n+2)}\frac{t^n}{n!}-\sum_{n=0}^\infty\frac{n(a-1,n+1)(b,n+1)}{(p,n+1)}\frac{t^{n}}{n!}\nonumber\\
&=-\alpha \sum_{n=0}^\infty\frac{(a-1,n+1)(b,n+1)}{(p,n+2)}\frac{t^n}{n!}.
 \end{align}
Hence, it follows from (\ref{(1-x)^2f''}),(\ref{h(t)}), (\ref{(1-x)^2F}), and (\ref{(1-t)F}) that
\begin{align*}
\frac{(1-x)^2 f''(x)}{1-t}&>\frac{d^2}{c^2}(1-t)\sum_{n=0}^\infty\frac{(u-1,n+1)(v,n+1)}{(p,n+2)}\Big[(\frac cd-1)(p+n+1)\nonumber\\
&\quad+(u+n)(v+n+1)-n(p+n+1)\Big]\frac{t^n}{n!}
-\alpha \sum_{n=0}^\infty\frac{(a-1,n+1)(b,n+1)}{(p,n+2)}\frac{t^n}{n!}\nonumber\\
&=\frac{d^2}{c^2}\sum_{n=1}^\infty\frac{(a-1,n)(b,n)}{(p,n+1)}\frac{t^{n-1}}{(n-1)!}\Big\{-\frac{\alpha c^2}{d^2}\nonumber\\
&\quad+\frac{(u-1,n)(v,n)}{(a-1,n)(b,n)}\Big[(\frac cd-1)(p+n)+(u+n-1)(v+n)-(n-1)(p+n)\Big]\Big\}\nonumber\\
&=\frac{d^2}{c^2}\sum_{n=1}^\infty\frac{(a,n-1)(b+1,n-1)}{(p,n+1)(n-1)!}\Big\{\frac{\alpha \beta c^2}{d^2}
+\frac{\Gamma(a)\Gamma(b+1)}{\Gamma(u-1)\Gamma(v)}Q(n)\Big\}t^{n-1},\nonumber\\
\end{align*}
where $Q(n)$ is defined as in Lemma \ref{Q(n)}. Since $u-1=a-1-\delta \in(-1,0)$, $\Gamma(u-1)<0$,  it follows from Lemma \ref{Q(n)} that
\begin{align}\label{f''2}
\frac{(1-x)^2 f''(x)}{1-t}&>\frac{d^2}{c^2}\sum_{n=1}^\infty\frac{(a,n-1)(b+1,n-1)}{(p,n+1)(n-1)!}\Big\{\frac{\alpha \beta c^2}{d^2}
+\frac{\Gamma(a)\Gamma(b+1)}{\Gamma(u-1)\Gamma(v)}Q(1)\Big\}t^{n-1}\nonumber\\
&=\frac{d^2}{c^2}\sum_{n=1}^\infty\frac{(a,n-1)(b+1,n-1)}{(p,n+1)(n-1)!}\Big[\frac{\alpha \beta c^2}{d^2}
+v(u-1)((c/d-1)(p+1)+u(v+1))\Big]t^{n-1}\nonumber\\
&=g(x, y)\frac{d^2}{c^2}\sum_{n=1}^\infty\frac{(a,n-1)(b+1,n-1)}{(p,n+1)(n-1)!}t^{n-1},
\end{align}
where $x= c/d \in (0, (\beta +p)/k], y=f_3(\delta)=(b+\delta)(a-\delta-1)$, and
\begin{align*}
g(x,y)=y^2+((p+1)x-1)y+\alpha \beta  x^2,
\end{align*}
 since $\delta\in (a-1,\delta_1]$, $y\in (-\beta x, 0]$ by Lemma \ref{lemmaf},
it follows from Lemma \ref{g(x,y)} that $g(x,y)\geq 0$ for $(x,y)\in D$, where $D$ is as Lemma \ref{g(x,y)}.

Hence, it follows from (\ref{f''2}) that $f''(x)>0$ for all $x\in (0,1)$, which shows that $f'(x)$ is strictly increasing on $(0,1)$, and so is $G_1(x)$ by (\ref{G1}), (\ref{f'}) and Lemma \ref{monotone}. Moreover, by L'H\^apital's rule,  we get
\begin{align}\label{G(1)}
G(1^{-})=G_1(0^{+})=f'(0)=\frac d{pc}\left((c/d-1)\alpha +(a-b-1)\delta-\delta^2\right)=C_1(\delta)
\end{align}
for $\delta\in(a-1,\delta_1], C_1(\delta)\geq C_1(\delta_1)=0$ and
\begin{align}\label{G(2)}
G(0^+)&=G_1(1^-)=f(1^-)=F(a-1-\delta,b+\delta;p;1)-F(a-1,b;p;1)\nonumber\\
&=\frac{1}{p B(a-\delta, b+1+\delta)}-\frac{1}{pB(a, b+1)}=C_2(\delta).
\end{align}
For part (2), we observe that, for $\delta_1<\delta<0$, the equations (\ref{G(1)}) and (\ref{G(2)}) hold again, both $C_1(\delta)$ and $C_2(\delta)$ are strictly decreasing from Lemma \ref{lemma0}(3), and $G(1^-)=C_1(\delta)<C_1(\delta_0)=0, G(0^+)=C_2(\delta)>C_2(0^-)=0$.

\begin{corollary}\label{main theorem2}
Let $a_0 $ be the minimum root of $4a(2-a)(1-a)^2(a^2-2a+2)^2=1$ in $(0,1)$. For $a\in [a_0,1)$, $0< c/d\leq [(a-1)^2+1]/[2(a-1)^2+1]$, and $\delta_2=(\sqrt{ c/d}-1)(1-a)<0$, we have

(1)  $a_0\in(1/32,1/31)$,

(2) If $\delta\in (a-1,\delta_2]$, then
the function
\begin{align*}
G_1(x)=\frac{F(a-1-\delta,1-a+\delta;1;1-x^d)-F(a-1,1-a;1;1-x^c)}{1-x^c}
\end{align*}
is strictly decreasing from (0,1) onto $(C_3(\delta),C_4(\delta))$,  where
\begin{align*}
&C_3(\delta)=-\frac dc\left(\delta^2+2(1-a)\delta+(1-a)^2(1-\frac cd)\right)\geq 0 \\
&C_4(\delta)=\frac{1}{B(a-\delta,2-a+\delta)}-\frac{1}{B(a,2-a)}=\frac 1\pi\left\{\frac{\sin(\pi(a-\delta))}{1-a+\delta}-\frac{\sin(\pi a)}{1-a}\right\}.
\end{align*}
In particular, for all $x\in(0,1)$,  if $\delta\in (a-1,\delta_2]$,
\begin{align}\label{inequalities corollary}
F(a-1,1-a;1;1-x^c)+C_3(\delta)(1-x^c)&<F(a-1-\delta,1-a+\delta;1;1-x^d)\nonumber\\
&<F(a-1,1-a;1;1-x^c)+C_4(\delta)(1-x^c).
\end{align}

(3) If $\delta_2<\delta<0$, then, as the functions of $x$, $F(a-1-\delta,1-a+\delta;1;1-x^d)$ and $F(a-1,1-a;1,1-x^c)$ are not directly comparable on $(0,1)$, that is, neither
\begin{align*}
F(a-1,1-a;1,1-x^c)<F(a-1-\delta,1-a+\delta;1;1-x^d)
\end{align*}
nor its reversed inequality holds for all $x\in(0,1)$.
\end{corollary}
\begin{proof}
(1) Part (1) follows from Lemma \ref{a_0,a_1}.

(2) Let $b=1-a$,  for $\alpha\geq \sqrt{3}\beta \Leftrightarrow $
\begin{align*}
 \frac{\sqrt3}{a}-\frac{1}{1-a}\leq 1+\sqrt3 \Leftrightarrow a\in (1-\frac1{\sqrt{1+\sqrt3}},1],
\end{align*}
for $\alpha<\sqrt3\beta,\  4h(\beta +p)\geq p^4 \Leftrightarrow$
\begin{align*}
 \frac{\sqrt3}{a}-\frac{1}{1-a}&> 1+\sqrt3,4a(2-a)(1-a)^2(a^2-2a+2)^2\geq 1\nonumber\\
 & \Leftrightarrow a\in (0,1-\frac1{\sqrt{1+\sqrt3}})\cap (a_0,a_1)\Leftrightarrow (a_0,1-\frac1{\sqrt{1+\sqrt3}}),
\end{align*}
where  $a_0, a_1 $ are as Lemma \ref{a_0,a_1}. By Theorem $\ref{main theorem1}$, if $a\in(a_0,1)$, $0< c/d\leq [(a-1)^2+1]/[2(a-1)^2+1]$, and $\delta_1=(\sqrt{ c/d}-1)(1-a)<0$, the inequality (\ref{inequalities corollary}) holds.

Part (3) follows from Theorem \ref{main theorem1}(2).
\end{proof}

\begin{remark}
The following results, which have been proved in \cite{Song},  can be directly obtained by Corollary \ref{main theorem2}.

 (I) Let  $a=b=1/2$,  $\alpha=a(b+1)=3/4, \beta=b(1-a)=1/4,$ hence $\alpha>\sqrt3 \beta$. We have

(1) Let $0< c/d\leq 5/6$, and $\delta_3=(\sqrt{ c/d}-1)/2<0$. Then,  if $\delta\in (-1/2,\delta_3]$, the following inequality holds for all  $x\in(0,1)$,
\begin{align*}
F(1/2,1/2;1;1-x^c)+C_3(\delta)(1-x^c)&<F(-1/2-\delta,1/2+\delta;1;1-x^d)\nonumber\\
&<F(-1/2,1/2;1;1-x^c)+C_4(\delta)(1-x^c).
\end{align*}
where
\begin{align*}
&C_5(\delta)=-\frac dc\left(\delta^2+\delta+\frac14(1-\frac cd)\right)\geq 0 \\
&C_6(\delta)=\frac{1}{B(1/2-\delta,3/2+\delta)}-\frac{2}{\pi}=\frac 2\pi\left[\frac{\cos(\pi\delta)}{1+2\delta}-1\right].
\end{align*}

(2) If  $0< c/d\leq5/6$, then
 \begin{align*}
\sup&\{\delta\in(-1/2,0)|=F(1/2,1/2;1;1-x^c)<F(-1/2-\delta,1/2+\delta;1;1-x^d), \\
&\hbox{for all}\  x\in (0,1)\}=(\sqrt{ c/d}-1)/2.
\end{align*}

(II) Let $c=2,d=3$ and  ${a}_0 $ be the minimum root of $4a(2-a)(1-a)^2(a^2-2a+2)^2=1$. For $a\in ({a}_0,1]$, and $\delta_4=(\sqrt{6}/3-1)(1-a)<0$, we have that:

 (1) If $\delta\in (a-1,\delta_4]$, the following inequality hold for all  $x\in(0,1)$,
 \begin{align*}
F(a-1,1-a;1;1-x^2)+C_7(\delta)(1-x^c)&<F(a-1-\delta,1-a+\delta;1;1-x^3)\nonumber\\
&<F(a-1,1-a;1;1-x^2)+C_8(\delta)(1-x^c).
\end{align*}
where
\begin{align*}
&C_7(\delta)=-\frac 32\left(\delta^2+2(1-a)\delta+\frac{(1-a)^2}{3}\right)\geq 0 \\
&C_8(\delta)=\frac{1}{B(a-\delta,2-a+\delta)}-\frac{1}{B(a,2-a)}=\frac 1\pi\left[\frac{\sin(\pi(a-\delta))}{1-a+\delta}-\frac{\sin(\pi a)}{1-a}\right].
\end{align*}
(2) If  $a\in [{a}_0,1)$, then
 \begin{align*}
\sup&\{\delta\in(-1/2,0)|=F(a-1,1-a;1;1-x^2)<F(a-1-\delta,1-a+\delta;1;1-x^3), \\
&\hbox{for all}\  x\in (0,1)\}=(\sqrt{6}/3-1)(1-a).
\end{align*}

\end{remark}

\clearpage

\end{document}